\documentclass[10pt]{amsart}
\usepackage{amsthm}
\usepackage{graphicx}
\usepackage{amssymb}
\usepackage{color}
\usepackage{tikz}
\usepackage{amsmath}
\usepackage{hyperref}
\usepackage{float}
\usepackage{enumerate}
\usepackage{tikz-cd}

\newtheorem{theorem}{Theorem}[section]

\newtheorem{conjecture}[theorem]{Conjecture}
\newtheorem{proposition}[theorem]{Proposition}

\newtheorem{remark}[theorem]{Remark}

\newcommand{\PP}{{\mathbb P}} 
\newcommand{\QQ}{\mathbb{Q}} \newcommand{\CC}{{\mathbb C}}
\newcommand{\ZZ}{{\mathbb Z}}

\newcommand{\cO}{\mathcal{O}}

\newcommand{\vol}{\operatorname{vol}}

\newcommand{\lct}{\mathrm{lct}}

\title{Quantized volume comparison for Fano manifolds}

\begin{document}

\author[K. Zhang]{Kewei Zhang}
\address{School of Mathematical Sciences, Beijing Normal University, Beijing, 100875, People's Republic of China.}
\email{kwzhang@bnu.edu.cn}

\begin{abstract}
A result of Kento Fujita says that the volume of a K\"ahler--Einstein Fano manifold is bounded from above by the volume of the projective space. In this short note we establish quantized versions of Fujita's result.
\end{abstract}

\maketitle

\setcounter{tocdepth}{1}


\section{Main results}

Let $X$ be a Fano manifold of dimension $n$.

It is shown by K. Fujita \cite{Fuj18-volume} that if $X$ admits K\"ahler--Einstein metric, then it satisfies the volume inequality:
$$
\vol(-K_X)\leq \vol(-K_{\PP^n})=(n+1)^n,
$$
and the equality holds if and only if $X\cong \PP^n$. Here we recall that the volume $\vol(-K_X)$ is defined as
$$
\vol(-K_X):=\lim_{m\to\infty}\frac{\dim H^0(X,-mK_X)}{m^n/n!},
$$
which is also equal to the intersection number $(-K_X)^n$. In differential geometric sense, it is also equal to the volume $\int_X\omega^n$ for any K\"ahler form $\omega\in c_1(-K_X)$.

Fujita actually proved something stronger in \cite{Fuj18-volume}. The same result holds if $X$ is only K-semistable, an algebro-geometric notion going back to \cite{Tian97,Don02} that is slightly weaker than the existence of a K\"ahler--Einstein metric. See \cite{Liu18,Zha20-volume} for further developments in this direction.

The first result of this note is a quantized version of Fujita's volume comparison for K-semistable Fano manifolds.

\begin{theorem}
\label{thm:Kss-vol-comp}
Assume that $X$ is a K-semistable Fano manifold, then there exists $m_0>0$ depending only on $n$ such that
$$
\dim H^0(X,-mK_X)\leq \dim H^0(\PP^n,-mK_{\PP^n})\text{ for all }m\geq m_0.
$$
If the equality holds for some $m\geq m_0$, then $X\cong \PP^n$.
\end{theorem}

This is likely well-known to experts. See \S \ref{sec:Kss} for a short proof that builds on \cite{KMM92} and \cite{Fuj18-volume}.

The above result motivates the following:
\begin{conjecture}
\label{conj:quantized-vol}
Let $X$ be a K-semistable Fano manifold. Then
$$
\dim H^0(X,-mK_X)\leq \dim H^0(\PP^n,-mK_{\PP^n})\text{ for any }m\geq 1.
$$
If the equality holds for some $m$, then $X\cong\PP^n$.
\end{conjecture}

At the moment the author is not aware of any example that could violate this conjecture.
On the other hand, it is easy to show that
this conjecture holds when $n\leq 3$ by Riemann--Roch (see \S \ref{sec:Kss}). However for higher dimensions it seems that substantially new ideas are needed to attack this conjecture. Even in the toric case this seems to be a tough problem.

Next, we present another quantized version of Fujita's volume comparison theorem. In fact, what Fujita actually proved in \cite{Fuj18-volume} is that 
$$\vol(-K_X)\leq \vol(-K_{\PP^n}) \text{ if } \delta(-K_X)\geq 1,$$
where $\delta(-K_X)$ denotes the delta invariant of $-K_X$ introduced in \cite{FO18,BJ17}. By the Fujita--Li criterion \cite{Li15,Fuj19}, $X$ is K-semistable if and only if $\delta(-K_X)\geq 1$.
This invariant also has its quantized version, $\delta_m(-K_X)$, that was also introduced in \cite{FO18}. We will recall the definitions of these invariants in \S \ref{sec:delta}.

The second result of this note is a quantized volume comparison theorem for Fano manifolds in terms of the $\delta_m$-invariant.

\begin{theorem}
\label{thm:delta-m-vol-comp}
     Let $X$ be a Fano manifold. Assume that for some $m\geq 1$ we have $\delta_{m}(-K_X)\geq 1$. Then
    $$
   \dim H^0(X,-mK_X)\leq \dim H^0(\PP^n,-mK_{\PP^n}),
    $$
    and the equality holds if and only if $X\cong \PP^n$.
\end{theorem}

Actually in \S \ref{sec:delta} we will prove a stronger result (Theorem \ref{thm:quant-vol-comp}) that only involves the local $\delta_m$-invariant at a point (cf. \cite[\S 2.2]{AZ20}). Our proof, although follows the same strategy as in \cite{Fuj16}, are presented in a more elementary way that only uses basic linear algebra. Our approach also reveals some new information about linear systems of a line bundle that could be of independent interest.

\textbf{Acknowledgment.} The author is grateful to Chen Jiang for helpful discussions. Thanks also go to the anonymous referee for the suggestions that sharpen this work. The author is supported by NSFC grants 12271038 and 12271040.






\section{Quantized volume comparison for K-semistable Fanos}
\label{sec:Kss}
First, we give a quick proof for Theorem \ref{thm:Kss-vol-comp}.

\begin{theorem}
Assume that $X$ is a K-semistable Fano manifold, then there exists $m_0>0$ depending only on $n$ such that
$$
\dim H^0(X,-mK_X)\leq \dim H^0(\PP^n,-mK_{\PP^n})\text{ for all }m\geq m_0.
$$
If the equality holds for some $m\geq m_0$, then $X\cong \PP^n$.
\end{theorem}

\begin{proof}
    By the Hirzebruch--Riemann-Roch formula,
    $$
    h^0(X,-mK_X):=\dim H^0(X,-mK_X)=\sum_{i=0}^n a_i(X)m^i,\ \text{for all }m\geq 0.
    $$
   Here $a_i(X)$ are topological constants depending only on the Chern classes of the holomorphic tangent bundle $T_X$. Some coefficients have simple expressions:
    $$
    a_n(X)=\frac{\vol(-K_X)}{n!},\ a_{n-1}(X)=\frac{\vol(-K_X)}{2(n-1)!},\ a_0(X)=\chi (\cO_X)=1.
    $$
    The other coefficients are more complicated to express. However, thanks to the boundedness of Fano manifolds \cite{KMM92}, they are all bounded:
    $$
    |a_i(X)|\leq A,
    $$
    where $A>0$ depends only on $n$. So $\vol(-K_X)$, as an integer, only takes finitely many different values.

    According to K. Fujita's estimate \cite{Fuj18-volume}, the leading coefficient satisfies
    $$
    a_n(X)\leq a_n(\PP^n),
    $$
    with the equality taking place only when $X\cong \PP^n$. 
    So it suffices to deal with the case where
    $$
    a_n(X)< a_n(\PP^n).
    $$
    Then we can find a sufficiently large integer $m_0>0$, depending only on $n$, such that
    $$
    h^0(X,-mK_X)< h^0(\PP^n,-mK_{\PP^n})
    $$
    for all $m\geq m_0$. This completes the proof.
\end{proof}

\begin{remark}
    A similar result holds for K-semistable $\QQ$-Fano varieties as well. But in this note we shall simply content ourselves with the smooth setting.
\end{remark}

Next, we show that Conjecture \ref{conj:quantized-vol} holds for $n\leq 3$, even without the K-semistability assumption (we are grateful to the referee for this point).

\begin{proposition}
    Let $X$ be a Fano manifold with dimension $n\leq 3$, then
    $$
h^0(X,-mK_X)\leq h^0(\PP^n,-mK_{\PP^n})\text{ for any }m\geq 1.
$$
If the equality holds for some $m$, then $X\cong\PP^n$.
\end{proposition}

\begin{proof}
When $n=1$ there is nothing to argue.
For $n=2$, the Hirzebruch--Riemann-Roch formula gives
$$
 h^0(X,-mK_X)=\frac{m(m+1)}{2}\vol(-K_X)+1,\ m\geq1.
$$
    For Fano 3-folds, combining the Hirzebruch--Riemann-Roch formula with the basic identity $c_1(X)c_2(X)=24$ one has
    $$
    h^0(X,-mK_X)=\frac{m(2m+1)(m+1)}{12}\vol(-K_X)+2m+1,\ m\geq1.
    $$
    It is well known that $\vol(-K_X)\leq \vol(-K_{\PP^n})$ holds for Fano manifolds with $n\leq 3$, and the equality holds only when $X\cong \PP^n$. So we conclude.
\end{proof}

The above argument breaks down when $n\geq 4$. For instance, the Hirzebruch--Riemann-Roch formula for Fano 4-fold gives
$$
h^0(X,-mK_X)
=\frac{m^4+2m^3+m^2}{24}\vol(-K_X)+\frac{m^2+m}{24}c_1(X)^2\cdot c_2(X)+1.
$$
The term $c_1(X)^2\cdot c_2(X)$ can be bounded by the Riemann-Roch type inequalities \cite{KM}, which however is probably not effective enough for our purpose. 

But we have the following partial answer when $n=4$.

\begin{proposition}
    Let $X$ be a K-semistable Fano manifold with $\dim X=4$. Assume that
    $
    h^0(X,-K_X)\leq h^0(\PP^4,-K_{\PP^4})=126$.
    Then
    $$
h^0(X,-mK_X)\leq h^0(\PP^4,-mK_{\PP^4})\text{ for any }m\geq 1.
$$
If the equality holds for some $m\geq 2$, then $X\cong \PP^4$.
\end{proposition}

\begin{proof}
    Using
    $$
    h^0(X,-K_X)=\frac{\vol(-K_X)}{6}+\frac{c_1(X)^2\cdot c_2(X)}{12}+1\leq 126,
    $$
    we derive that
    $$
   c_1(X)^2\cdot c_2(X)\leq 1500-2\cdot \vol(-K_X).
    $$
    Therefore,
    \begin{equation*}
        \begin{split}
           h^0(X,-mK_X)&\leq \frac{m^4+2m^3+m^2}{24}\vol(-K_X)+\frac{m^2+m}{24}(1500-2\cdot \vol(-K_X))+1.\\
        \end{split}
    \end{equation*}
   Note that
    $$
    h^0(\PP^4,-mK_{\PP^4})
=\frac{m^4+2m^3+m^2}{24}\cdot 625+\frac{m^2+m}{24}\cdot 250+1,
    $$
    so
    $$
    h^0(X,-mK_X)-h^0(\PP^4,-mK_{\PP^4})\leq (\vol(-K_X)-625)\frac{m^4+2m^3-m^2-2m}{24}.
    $$
    Then Fujita's volume comparison \cite{Fuj18-volume} allows us to conclude.
\end{proof}

Therefore, to prove Conjecture \ref{conj:quantized-vol} for Fano 4-fold, it suffices to show that $h^0(X,-K_X)< 126$ when $X\not\cong \PP^4$, which is indeed the case when $X$ has Picard number 1 (see the estimates in \cite{H03}).


For higher dimensions, more complicated Chern numbers will be involved in the Hirzebruch--Riemann-Roch formula, so probably a different approach is needed to attack the conjecture.

\section{The quantized delta invariant}
\label{sec:delta}

We now recall the definition of $\delta$-invariant, following \cite{FO18,BJ17,AZ20}.

Let $X$ be a compact complex manifold of dimension $n$ and let $L$ be a holomorphic line bundle on $X$.
For $m\geq 1$, we let
$$
d_m:=h^0(X,mL):=\dim_\CC H^0(X,mL)
$$
be the dimension of the space of global holomorphic sections of $mL$.
Assume that $d_m>0$ and let $\{s_1,...,s_{d_m}\}$ be a basis of $H^0(X,mL)$ and let
$$
D=\frac{1}{md_m}\sum_{i=1}^{d_m}(s_i=0)
$$
be the $\QQ$-divisor associated with this basis, which will be called an \emph{$m$-basis divisor} of $L$. For any point $p\in X$, define
$$
\delta_{m,p}(L):=\inf\{\lct_p(X,D)|\text{ $D$ is an $m$-basis divisor}\}.
$$
Here $\lct_p(X,D)$ denotes the log canonical threshold of $D$ at $p$. 
In analytic terms,
$$
\lct_p(X,D)=\sup\left\{\lambda>0:\bigg(\prod_{i=1}^{d_m}|s_i|^2\bigg)^{\frac{-\lambda}{md_m}}\text{ is integrable near }p\right\}.
$$
If the vanishing order of $D$ at $p$ is larger than $n$, then $\lct_p(X,D)<1$ (see \cite[Lemma 8.10]{Kol97}), meaning that the pair $(X,D)$ is not log canonical at $p$.

The $\delta_m$-invariant of $L$ is then defined as
$$
\delta_m(L)=\inf_{p\in X}\delta_{m,p}(L).
$$
If $X$ is projective and $L$ is a big line bundle on $X$, then we further put
$$
\delta(L):=\lim_{m\to\infty}\delta_m(L).
$$
Note that the above limit exists by Blum--Jonsson \cite{BJ17}.

Now we move on to the following estimate.

\begin{proposition}
\label{prop:delta-m-vol-comp}
    Assume that $\delta_{m,p}(L)\geq 1$ for some $m\geq 1$ and $p\in X$. Then
    $$
    h^0(X,mL)\leq h^0(\PP^n,-mK_{\PP^n}).
    $$
\end{proposition}

\begin{proof}
The original idea is contained in \cite[Theorem 2.3]{Fuj18-volume}. Here we present a proof that uses only linear algebra.
    For $j\geq 1$ let
    $
    \frak m_{p}^j
    $
    be the ideal sheaf generated by holomorphic functions with vanishing order at least $j$ at $p$. 
    

First, observe that
\begin{equation}
    \label{eq:h0-jP-estimate}
    h^0(X,mL)-h^0(X,\frak m_{p}^j\otimes mL)\leq \binom{n+j-1}{n}.
\end{equation}
Indeed, $
H^0(X,mL)/H^0(X,\frak m_{p}^j\otimes mL)
$ is spanned by sections of $mL$ with vanishing order at $p$ less than $j$. By considering Taylor expansions of these sections at $p$, the dimension of this quotient space cannot be bigger than the dimension of the space of $n$-variable polynomials with degree less than $j$. So the assertion follows.

To show that $
    h^0(X,mL)\leq h^0(\PP^n,-mK_{\PP^n})
    $, we argue by contradiction. Suppose otherwise that
$$
h^0(X,mL)> h^0(\PP^n,-mK_{\PP^n})=\binom{n+m(n+1)}{n}.
$$
This together with \eqref{eq:h0-jP-estimate} then imply that
\begin{equation*}
    \begin{split}
        \frac{\sum_{j=1}^{m(n+1)+1}h^0(X,\frak m_{p}^j\otimes mL)}{m h^0(X,mL)}&\geq \frac{mn+m+1}{m}-\frac{\sum_{j=1}^{m(n+1)+1}\binom{n+j-1}{n}}{mh^0(X,mL)}\\
        &=n+1+\frac{1}{m}-\frac{\binom{m(n+1)+n+1}{n+1}}{mh^0(X,mL)}\\
        &> n+1+\frac{1}{m}-\frac{\binom{m(n+1)+n+1}{n+1}}{m\binom{m(n+1)+n}{n}}\\
        &=n.\\
    \end{split}
\end{equation*}
Using linear algebra as in \cite[Lemma 2.2]{FO18} and \cite[Lemma 2.7]{CRZ19} then yields an $m$-basis divisor $D$ of $L$ such that its vanishing order at $p$ is larger than $n$. So the pair $(X,D)$ cannot be log canonical at $p$. Thus $\delta_{m,p}(L)<1$, a contradiction.

\end{proof}

From the above proof, one can also easily characterize the equality case.

\begin{proposition}
\label{prop:delta-m-vol-eq-case}
      Assume that $\delta_{m,p}(L)\geq 1$ for some $m\geq 1$ and $p\in X$. If it happens that
$$
h^0(X,mL)= h^0(\PP^n,-mK_{\PP^n})=\binom{m(n+1)+n}{n},
$$
then the linear system $|mL|$ separates $m(n+1)$-jets at $p$. More precisely, we can find a holomorphic coordinate system $(z_1,...,z_n)$ around $p$ and a basis $\{s_\alpha\}$ of $H^0(X,mL)$ such that the local expressions of $s_\alpha$ are monomials modulo higher order terms:
$$
s_\alpha(z)=z_1^{\alpha_1}\cdots z_n^{\alpha_n}+o(|z|^{m(n+1)}),
$$
where $\alpha=(\alpha_1,...,\alpha_n)$ runs through all the indices such that
$$
\alpha_i\in\ZZ_{\geq 0},\ \sum_{i=1}^n \alpha_i\leq m(n+1).
$$
\end{proposition}

\begin{proof}
    From the proof of Proposition \ref{prop:delta-m-vol-comp}, the equality case will force \eqref{eq:h0-jP-estimate} to be an equality for all $1\leq j\leq  m(n+1)+1$. Namely,
    $$
     h^0(X,mL)-h^0(X,\frak m_{p}^j\otimes mL)=\binom{n+j-1}{n},\ 1\leq j \leq  m(n+1)+1.
    $$
    So in particular,
    $$
    h^0(X,\frak {m}_{p}^{m(n+1)+1} \otimes mL)=0,
    $$
    meaning that the vanishing order at $p$ of any nonzero section of $mL$ is less than $m(n+1)+1$.
   Since $h^0(X,mL)=\binom{m(n+1)+n}{n}$ is exactly equal to the dimension of the space of $n$-variable polynomials with degree less than $m(n+1)+1$, we conclude.
\end{proof}

As a consequence, we obtain the following quantized volume comparison theorem, which implies Theorem \ref{thm:delta-m-vol-comp}.

\begin{theorem}
\label{thm:quant-vol-comp}
    Let $X$ be a Fano manifold. Assume that for some $m\geq 1$ and $p\in X$ we have $\delta_{m,p}(-K_X)\geq 1$. Then
    $$
    h^0(X,-mK_X)\leq h^0(\PP^n,-mK_{\PP^n}),
    $$
    and the equality holds if and only if $X\cong \PP^n$.
\end{theorem}

\begin{proof}

In view of Proposition \ref{prop:delta-m-vol-comp}, it suffices to deal with the equality case.

If $X\cong\PP^n$, then $\delta_{m}(-K_X)=\delta_{m,p}(-K_X)$ for any $p\in \PP^n$, as $\PP^n$ is homogeneous. Moreover, by \cite[Corollary 7.2]{RTZ20} (see also \cite[Corollary A.8]{JR25}), we have $\delta_{m}(-K_{\PP^n})=1$. So $X\cong \PP^n$ indeed attains the equality case.

Now suppose that $
    h^0(X,-mK_X)=h^0(\PP^n,-mK_{\PP^n}).
    $   
By
Proposition \ref{prop:delta-m-vol-eq-case}, the linear system $|-mK_X|$ generates $m(n+1)$-jets at $p$. By \cite[Theorem 1]{BS09}, the Seshadri constant of $-K_X$ at $p$ is no less than $n+1$, so $X\cong \PP^n$ by \cite[Theorem 2]{BS09}.
\end{proof}

Replacing \cite{BS09} with the more general characterization of $\PP^n$ due to Liu--Zhuang \cite{LZ18}, one can actually prove the following.

\begin{theorem}
    Let $X$ be a $\QQ$-Fano variety with klt singularities. Let $m\geq 1$ be such that $-mK_X$ is Cartier. Assume that for some smooth point $p\in X$ we have $\delta_{m,p}(-K_X)\geq 1$. Then
    $$
    h^0(X,-mK_X)\leq h^0(\PP^n,-mK_{\PP^n}),
    $$
    and the equality holds if and only if $X\cong \PP^n$.
\end{theorem}

We omit the proof since it is the same as the one for Theorem \ref{thm:quant-vol-comp}.

\bibliography{ref.bib}

\begin{thebibliography}{KMM92}

\bibitem[AZ22]{AZ20}
Hamid Abban and Ziquan Zhuang.
\newblock K-stability of {Fano} varieties via admissible flags.
\newblock {\em Forum Math. Pi}, 10:43, 2022.
\newblock Id/No e15.

\bibitem[BJ20]{BJ17}
Harold Blum and Mattias Jonsson.
\newblock Thresholds, valuations, and {K}-stability.
\newblock {\em Adv. Math.}, 365:107062, 57, 2020.

\bibitem[BS09]{BS09}
Thomas Bauer and Tomasz Szemberg.
\newblock Seshadri constants and the generation of jets.
\newblock {\em Journal of Pure and Applied Algebra}, 213(11):2134--2140, 2009.

\bibitem[CRZ19]{CRZ19}
Ivan~A. Cheltsov, Yanir~A. Rubinstein, and Kewei Zhang.
\newblock Basis log canonical thresholds, local intersection estimates, and asymptotically log del {P}ezzo surfaces.
\newblock {\em Selecta Math. (N.S.)}, 25(2):25:34, 2019.

\bibitem[Don02]{Don02}
S.~K. Donaldson.
\newblock Scalar curvature and stability of toric varieties.
\newblock {\em J. Differential Geom.}, 62(2):289--349, 2002.

\bibitem[FO18]{FO18}
Kento Fujita and Yuji Odaka.
\newblock On the {K}-stability of {F}ano varieties and anticanonical divisors.
\newblock {\em Tohoku Math. J. (2)}, 70(4):511--521, 2018.

\bibitem[Fuj16]{Fuj16}
Kento Fujita.
\newblock On {$K$}-stability and the volume functions of {$\Bbb{Q}$}-{F}ano varieties.
\newblock {\em Proc. Lond. Math. Soc. (3)}, 113(5):541--582, 2016.

\bibitem[Fuj18]{Fuj18-volume}
Kento Fujita.
\newblock Optimal bounds for the volumes of {K}\"{a}hler-{E}instein {F}ano manifolds.
\newblock {\em Amer. J. Math.}, 140(2):391--414, 2018.

\bibitem[Fuj19]{Fuj19}
Kento Fujita.
\newblock A valuative criterion for uniform {K}-stability of {$\Bbb{Q}$}-{F}ano varieties.
\newblock {\em J. Reine Angew. Math.}, 751:309--338, 2019.

\bibitem[Hwa03]{H03}
Jun-Muk Hwang.
\newblock On the degrees of {Fano} four-folds of {Picard} number 1.
\newblock {\em J. Reine Angew. Math.}, 556:225--235, 2003.

\bibitem[JR25]{JR25}
Chenzi Jin and Yanir~A. Rubinstein.
\newblock Asymptotics of quantized barycenters of lattice polytopes with applications to algebraic geometry.
\newblock {\em Math. Z.}, 309(3):41, 2025.
\newblock Id/No 49.

\bibitem[KM83]{KM}
J.~Kollár and T.~Matsusaka.
\newblock Riemann-roch type inequalities.
\newblock {\em American Journal of Mathematics}, 105(1):229--252, 1983.

\bibitem[KMM92]{KMM92}
J\'{a}nos Koll\'{a}r, Yoichi Miyaoka, and Shigefumi Mori.
\newblock Rational connectedness and boundedness of {F}ano manifolds.
\newblock {\em J. Differential Geom.}, 36(3):765--779, 1992.

\bibitem[Kol97]{Kol97}
J\'{a}nos Koll\'{a}r.
\newblock Singularities of pairs.
\newblock In {\em Algebraic geometry---{S}anta {C}ruz 1995}, volume~62 of {\em Proc. Sympos. Pure Math.}, pages 221--287. Amer. Math. Soc., Providence, RI, 1997.

\bibitem[Li17]{Li15}
Chi Li.
\newblock K-semistability is equivariant volume minimization.
\newblock {\em Duke Math. J.}, 166(16):3147--3218, 2017.

\bibitem[Liu18]{Liu18}
Yuchen Liu.
\newblock The volume of singular {K}\"{a}hler-{E}instein {F}ano varieties.
\newblock {\em Compos. Math.}, 154(6):1131--1158, 2018.

\bibitem[LZ18]{LZ18}
Yuchen Liu and Ziquan Zhuang.
\newblock Characterization of projective spaces by {Seshadri} constants.
\newblock {\em Math. Z.}, 289(1-2):25--38, 2018.

\bibitem[RTZ21]{RTZ20}
Yanir~A. Rubinstein, Gang Tian, and Kewei Zhang.
\newblock Basis divisors and balanced metrics.
\newblock {\em J. Reine Angew. Math.}, 778:171--218, 2021.

\bibitem[Tia97]{Tian97}
Gang Tian.
\newblock K\"{a}hler-{E}instein metrics with positive scalar curvature.
\newblock {\em Invent. Math.}, 130(1):1--37, 1997.

\bibitem[Zha22]{Zha20-volume}
Kewei Zhang.
\newblock On the optimal volume upper bound for {K{\"a}hler} manifolds with positive {Ricci} curvature (with an appendix by {Yuchen} {Liu}).
\newblock {\em Int. Math. Res. Not.}, 2022(8):6135--6156, 2022.

\end{thebibliography}
\bibliographystyle{alpha}

\end{document}